\documentclass[12pt]{amsart}
\usepackage{amssymb,bbold}
\usepackage[all]{xy}

\theoremstyle{plain}
\newtheorem{theorem}{Theorem}
\newtheorem{corollary}[theorem]{Corollary}
\newtheorem{lemma}[theorem]{Lemma}
\newtheorem{proposition}[theorem]{Proposition}

\theoremstyle{definition}
\newtheorem{definition}[theorem]{Definition}
\newtheorem{remark}[theorem]{Remark}

\begin{document}
\baselineskip 18pt

\title[Unbounded Banach-Saks and unbounded Grothendieck operators ]
      {Unbounded Banach-Saks operators and unbounded Grothendieck operators on Banach lattices}

\author[O.~Zabeti]{Omid Zabeti}


\address[O.~Zabeti]
  {Department of Mathematics, Faculty of Mathematics, Statistics, and Computer science,
   University of Sistan and Baluchestan, Zahedan,
   P.O. Box 98135-674. Iran}
\email{o.zabeti@gmail.com}


\begin{abstract}
Suppose $E$ is a Banach lattice. Recently, there have been some motivating contexts regarding the known Banach-Saks property and the Grothendieck property from an order point of view. In this paper, we establish these results for operators that enjoy different types considered for the Banach-Saks property as well as for different notions related to the Grothendieck property. In particular, beside other results, we characterize order continuity and reflexivity of Banach lattices in terms of the corresponding bounded operators defined on them.
\end{abstract}

\date{\today}

\maketitle
\section{motivation and introduction}
Let us start with some motivation. Certainly, a Banach lattice is a Banach space with an order structure and of course with some suitable connections between ordered and analysis aspects. So, it is natural to expect some order aspects of Banach lattices and also operators between them. This line has been growing from the early stage of appearance of vector lattices (with the remarkable Riesz-Kantorovich formulae) until now, using different types of convergence with the aid of the order structure beside the analysis aspects, such as $uaw$-Dunford-Pettis ($M$-weakly compact) operators, unbounded continuous operators ($WM$-weakly compact operators) and so on. Recently, there have been some attention to the order structure for some notions that have been introduced and investigated initially in Banach spaces such as the Banach-Saks property and the Grothendieck property. Different types of the mentioned properties have been defined and some interesting results regarding characterizations of some known spaces in Banach lattice theory such as order continuity or reflexivity have been obtained (see \cite{AEH, DLS, EGZ, MFMA, Z2, Z1} for more details). In this paper, we consider operator versions regarding the different types considered for the Banach-Saks property and various concepts related to the Grothendieck property. As our main results, in both sections, we characterize order continuous Banach lattices and reflexive ones in terms of these different classes of bounded operators. All operators in this paper are assumed to be continuous. For a detailed exposition on Banach lattices and operators between them, as well as related topics, see \cite{AB, Ni, Z1}.
\section{Banach-Saks operators}
First, let us recall different types related to the Banach-Saks property for a Banach lattice. Suppose $E$ is a Banach lattice. $E$ is said to have the { unbounded Banach-Saks property} ({UBSP}, for short) if for every norm bounded $uaw$-null sequence $(x_n)\subseteq E$ (i.e. for each $u\in E_{+}$, $|x_n|\wedge u\xrightarrow{w}0$), there is a subsequence $(x_{n_k})$  whose Ces\`{a}ro means is convergent. Moreover, recall that $E$ possesses the disjoint Banach-Saks property ({DBSP}, for short) if every bounded disjoint sequence in $E$ has a Ces\`{a}ro convergent subsequence; $E$ has the disjoint weak Banach-Saks property ({DWBSP}, in brief) if every disjoint weakly null sequence in $E$ has a Ces\'{a}ro convergent subsequence. Furthermore, $E$ possesses the weak Banach-Saks property ({WBSP}, in brief) if for every weakly null sequence $(x_n)$, it has a subsequence which is Ces\`{a}ro convergent. Finally observe that $E$ possesses the Banach-Saks property ({BSP}) if every bounded sequence $(x_n)\subseteq E$ has a Ces\'{a}ro convergent subsequence. For a brief but comprehensive context in this subject, see \cite{GTX}. Also, for a detailed exposition on these notions, see \cite{Z1}.

Now, let us consider the operators versions of these types of the Banach-Saks properties.

\begin{definition}\label{31}
Suppose $E$ and $F$ are Banach lattices and $T:E\to F$ is a continuous operator. $T$ is said to be
\begin{itemize}
		\item[\em (i)] {the unbounded Banach-Saks operator ({\bf UBSO}, for short) if every  bounded $uaw$-null sequence $(x_n)\subseteq E$, $(T(x_n))$ has a Ces\`{a}ro convergent subsequence; that is, there exists a subsequence $(x_{n_k})$ of $(x_n)$ such that  $(\frac{1}{n}\Sigma_{k=1}^{n}T(x_{n_k}))$ is norm convergent.}
\item[\em (ii)]
	{The disjoint Banach-Saks operator ({\bf DBSO}, in notation) if every  bounded disjoint sequence $(x_n)\subseteq E$, $(T(x_n))$ has a  Ces\`{a}ro convergent subsequence.}
\item[\em (iii)]
	{The weak Banach-Saks operator ({\bf WBSO}, in notation) if every weakly null sequence $(x_n)\subseteq E$,  $(T(x_n))$ has a  Ces\`{a}ro convergent subsequence.}
\item[\em (iv)]	{ The disjoint weakly Banach-Saks operator ({\bf DWBSO}, in notation) if every disjoint weakly null sequence $(x_n)\subseteq E$, $(T(x_n))$ has a  Ces\`{a}ro convergent subsequence.}
		\item[\em (v)] {The Banach-Saks operator ({\bf BSO}, in brief) if every  bounded  sequence $(x_n)\subseteq E$, $(T(x_n))$ has a  Ces\`{a}ro convergent subsequence.}
		
			\end{itemize}
\end{definition}
The parts $(iii)$ and $(iv)$ were defined initially in the definition 6.16 from \cite{GTX}. Theses notions extend ideas considered in \cite[Section 4]{Z1} as a Banach lattice $E$ has the corresponding Banach-Saks property if and only if so is the identity operator on $E$. Moreover, note that while $E$ has the corresponding Banach-Saks property and $F$ is any Banach lattice, every continuous operator $T:E\to F$ possesses the Banach-Saks operator version, as described in Definition \ref{31}, respectively. Therefore, we conclude that these concepts are far from being equivalent, in general as described in \cite[Section 4]{Z1}.

Before anything, the following easy fact is handy and fruitful.
\begin{lemma}
Suppose $E$, $F$ and $G$ are Banach lattices and $T:E\to F$ and $S:F\to G$ are continuous operators. Then, we have the following facts.
\begin{itemize}
\item[\em (i)] {If $T$ is {\bf UBSO}, then so is $ST$}.
\item[\em (ii)]{If $T$ is {\bf DBSO}, then so is $ST$}.
\item[\em (iii)]{If $T$ is {\bf WBSO}, then so is $ST$}.
\item[\em (iv)]{If $T$ is {\bf DWBSO}, then so is $ST$}.
\item[\em (v)]{If $T$ is {\bf BSO}, then so is $ST$}.

\end{itemize}
\end{lemma}

Now, we extend \cite[Lemma 36]{Z1} to the operator case; the proof has essentially the same idea.
\begin{lemma}\label{1}
\begin{itemize}
\item[\em (i)] {Suppose $E'$ is order continuous. Then for every Banach lattice $F$, every {\bf WBSO} $T:E\to F$ is {\bf UBSO}}.
\item[\em (ii)]{For any Banach lattices $E$ and $F$, every {\bf UBSO} $T:E\to F$ is also {\bf DBSO}}.
\item[\em (iii)] {Suppose $E$ is either an $AM$-space or an atomic order continuous Banach lattice. Then for every Banach lattice $F$, every {\bf UBSO} $T:E\to F$ is {\bf WBSO}}.
\end{itemize}
\end{lemma}

For the converse, we have the following.
\begin{theorem}
Suppose $E$ is a Banach lattice. Then the following are equivalent.
\begin{itemize}
\item[\em (i)] { $E'$ is order continuous}.
\item[\em (ii)]{every {\bf WBSO} $T:E\to \ell_1$ is {\bf UBSO}}.
\end{itemize}
\end{theorem}
\begin{proof}
$(i)\to (ii)$ is trivial by Lemma \ref{1}. For the converse, suppose not. By \cite[Proposition 2.3.11]{Ni}, there exists a positive projection $P:E\to \ell_1$. $P$ is in fact a {\bf WBSO}; suppose $(x_n)$ is a weakly null sequence in $E$ so that $P(x_n)$ is weakly null in $\ell_1$. By the Schur property, it is norm null so that whose Ces\`{a}ro means is also norm null, as well. But $P$ is not an {\bf UBSO}. Suppose $(e_n)$ is the standard basis in $\ell_1$. It is disjoint in $\ell_1$ so that in $E$. By \cite[Lemma 2]{Z}, it is $uaw$-null. Nevertheless, it is easy to see that whose Ces\`{a}ro means is not convergent in $\ell_1$, as claimed.
\end{proof}

Moreover, for a converse of Lemma \ref{1}, part $(ii)$. we have the following.
\begin{proposition}\label{3}
Suppose $E$ is an order continuous Banach lattice and $F$ is any Banach lattice. Then, every {\bf DBSO} $T:E\to F$ is {\bf UBSO}.
\end{proposition}
\begin{proof}
Suppose $(x_n)$ is a norm bounded $uaw$-null sequence in $E$. By \cite[Theorem 4]{Z} and \cite[Theorem 3.2]{DOT}, there are a subsequence $(x_{n_k})$ of $(x_n)$ and a disjoint sequence $(d_k)$ such that $\|x_{n_k}-d_k\|\rightarrow 0$. By passing to a subsequence, we may assume that $\lim_{m\rightarrow \infty}\frac{1}{m}\Sigma_{i=1}^{m}T(d_{i})\rightarrow 0$. Now, the result follows from the following inequality.
\[\|\frac{1}{m}\Sigma_{i=1}^{m}T(x_{n_i})-\frac{1}{m}\Sigma_{i=1}^{m}T(d_i)\|\leq \frac{1}{m}\Sigma_{i=1}^{m}\|T\|\|x_{n_i}-d_i\|\rightarrow 0.\]
\end{proof}
\begin{remark}\label{32}
Note that order continuity is essential in Proposition \ref{3} and can not be omitted. Consider the identity operator $I$ on $\ell_{\infty}$; it is obviously {\bf DBSO} by \cite[Lemma 38]{Z1} but it fails to be {\bf UBSO}, the $uaw$-null sequence $(u_n)$ defined via $u_n=(0,\ldots,0,1,\ldots,1,0,\ldots)$, in which one is appeared $n$-times (from the $n$-th position until the $2n$-th position), does not have any Ces\`{a}ro convergent subsequence; although it is weakly null so that $uaw$-null by \cite[Theorem 7]{Z}, although it does have the {\bf DBSP}, certaily.
\end{remark}
For a sort of the converse of Proposition \ref{3}, we have the following partial results.
\begin{proposition}
Suppose $F$ is a Banach lattice. If every continuous operator $T:\ell_1\to F$ is {\bf UBSO} ({\bf DBSO}), then $F'$ is order continuous.
\end{proposition}
\begin{proof}
First note that since the domain $\ell_1$ has an order continuous norm, {\bf UBSO} is equivalent by {\bf DBSO} by Proposition \ref{3} and Lemma \ref{1}. Suppose not. By \cite[Theorem 4.71]{AB}, $E$ has a lattice copy of $\ell_1$. Consider the embedding map from $\ell_1$ into $F$; the standard disjoint sequence $(e_n)$ is $uaw$-null in both $\ell_1$ and $E$ by \cite[Lemma 2]{Z}, however, the Ces\`{a}ro means is not convergent for any subsequence.
\end{proof}
\begin{proposition}
Suppose $F$ is a $\sigma$-order complete Banach lattice. If every continuous operator $T:\ell_{\infty}\to F$ is {\bf UBSO}, then $F$ is order continuous.
\end{proposition}
\begin{proof}
 Suppose not. By \cite[Theorem 4.56]{AB}, $E$ has a lattice copy of $\ell_{\infty}$. Consider the embedding map from $\ell_{\infty}$ into $F$; now, consider Remark \ref{32} to derive the desired result.
\end{proof}
Now, we consider some ideal properties for these classes of operators. Recall that a continuous operator $T:E\to F$, where $E$ and $F$ are Banach lattices, is said to be unbounded continuous provided that it maps every bounded $uaw$-null sequence $(x_n)\subseteq E$, to a weakly null sequence; $T$ is $uaw$-continuous if it maps norm bounded $uaw$-null sequences into $uaw$-null ones. For more details, see \cite{Z1}.
\begin{proposition}\label{2}
 Let $E,F$ and $G$ be Banach lattices. Then the following assertions hold.
	\begin{itemize}
		\item[\em (i)] {If $T:F\to G$ is a {\bf DBSO} and $S:E\to F$ is disjoint-preserving then $TS$ is {\bf DBSO}}.
		\item[\em (ii)] {If $T:F\to G$ is an {\bf UBSO} operator and $S:E\to F$ is $uaw$-continuous, then $TS$ is {\bf UBSO}}.
		\item[\em (iii)]
	{ If $T:F\to G$ is {\bf WBSO} and $S:E\to F$ is an unbounded continuous operator, then $TS$ is also {\bf UBSO}}.
	\end{itemize}
\end{proposition}
\begin{proof}
$(i)$. Suppose $(x_n)$ is a norm bounded disjoint sequence in $E$. Therefore, by the assumption, $(S(x_n))$ is also bounded and disjoint in $F$. Therefore, $(TS(x_n))$ has a Ces\`{a}ro convergent subsequence.

$(ii)$. Suppose $(x_n)$ is a norm bounded $uaw$-null sequence in $E$. Therefore, $S(x_n)\xrightarrow{uaw}0$. Thus, $(TS(x_n))$ has a Ces\`{a}ro convergent subsequence.

$(iii)$. Suppose $(x_n)$ is a norm bounded $uaw$-null sequence in $E$. By the assumption, $S(x_n)\xrightarrow{w}0$ so that $(TS(x_n))$ has a Ces\`{a}ro convergent subsequence.
\end{proof}
\begin{theorem}
Suppose $E$ is a Banach lattice such that both $E$ and $E'$ have order continuous norms and $F$ is any Banach lattice. Then each {\bf DWBSO} $T:E\to F$ is {\bf UBSO}.
\end{theorem}
\begin{proof}
Suppose $T:E\to F$ is {\bf DWBSO} and $(x_n)$ is a norm bounded $uaw$-null sequence in $E$. By \cite[Theorem 4]{Z} and \cite[Theorem 3.2]{DOT}, there are a subsequence $(x_{n_k})$ of $(x_n)$ and a disjoint sequence $(d_k)$ such that $\|x_{n_k}-d_k\|\rightarrow 0$. By \cite[Lemma 2]{Z}, $x_n\xrightarrow{uaw}0$ so that $x_n\xrightarrow{w}0$ by \cite[Theorem 7]{Z}. By passing to a subsequence and by the assumption, we may assume that $\lim_{m\rightarrow \infty}\frac{1}{m}\Sigma_{i=1}^{m}T(d_{i})\rightarrow 0$. Now, the result follows from the following inequality.
\[\|\frac{1}{m}\Sigma_{i=1}^{m}T(x_{n_i})-\frac{1}{m}\Sigma_{i=1}^{m}T(d_i)\|\leq \frac{1}{m}\Sigma_{i=1}^{m}\|T\|\|x_{n_i}-d_i\|\rightarrow 0.\]
\end{proof}
For the converse, we have the following fact.
\begin{theorem}
Suppose $E$ is a $\sigma$-order complete Banach lattice for that every {\bf DWBSO} $T:E\to \ell_{\infty}$ is {\bf UBSO}. Then, both $E$ and $E'$ possess order continuous norms.
\end{theorem}
\begin{proof}
Suppose not. Observe that by \cite[Proposition 2.3.11]{Ni} and \cite[Theorem 2.4.12]{Ni}, there is a  positive projection $P$ from $E$ onto $\ell_1$ and a positive projection $Q$ from $E$ onto $c_0$. Moreover, assume that $\iota_1$ and $\iota_2$ are the inclusion maps from $\ell_1$ and $c_0$ into $\ell_\infty$, respectively. Consider operators $\iota_1 P$ and $\iota_2 Q$. Both of them are {\bf DWBSO}. For, assume that $(x_n)$ is a disjoint weakly null sequence in $E$ so that $P(x_n)$ is disjoint and norm null in $\ell_1$ (because of Schur property) and so in $\ell_{\infty}$. Therefore, by the assumption, $\iota_1 P(x_n)$ has a Ces\`{a}ro convergent subsequence. Moreover, note that $c_0$ possesses the {\bf DWBSP}; thus, we conclude that $\iota_2 Q(x_n)$ also possesses a Ces\`{a}ro convergent means. Nevertheless, neither $\iota_1 P$ nor $\iota_2 Q$ are {\bf UBSO}. For the former case, consider the standard sequence $(e_n)\subseteq \ell_1$ and for the latter case, consider $(u_n)\subseteq \ell_{\infty}$ considered in Remark \ref{32}. This would complete the proof.
\end{proof}
\section{The Grothendieck operators}
Recall that a Banach space $X$ has the Grothendieck property if for every sequence $({x'}_n)\subseteq X'$, we have ${x'}_n\xrightarrow{w^{*}}0$ implies that ${x'}_n\xrightarrow{w}0$. It is known that every reflexive space has the Grothendieck property. Moreover, in the setting of Banach lattice theory, every $\sigma$-order complete $AM$-space with unit has this property, as well (see \cite[Theorem 4.44]{AB}). A  Banach lattice $E$ is said to have the disjoint Grothendieck property if for every norm bounded disjoint sequence $({x'}_{n})\subseteq E'$, we have ${x'}_n\xrightarrow{w}0$. Now, we define an unbounded version of this property as follows. Suppose $E$ is a Banach lattice. $E$ is said to have the {\bf unbounded Grothendieck property} if for every bounded sequence $({x_n}')\subseteq E'$, ${x_n}'\xrightarrow{uaw^{*}}0$ implies that ${x_n}'\xrightarrow{w}0$. It is easy to see that every reflexive space possesses the unbounded Grothendieck property, as well. In spite of \cite[Theorem 4.44]{AB}, $\ell_{\infty}$ does not have the unbounded Grothendieck property; see \cite[Lemma 3]{Z2}. Moreover, we mention the main result from \cite{Z2} as follows.
\begin{theorem}
A Banach lattice $E$ is reflexive if and only if it possesses the unbounded Grothendieck property.
\end{theorem}

Now, consider the following definitions regarding operators that enjoy different types considered for the Grothendieck property.

\begin{definition}
Suppose $E$ and $F$ are Banach lattices and $T:E\to F$ is a continuous operator. $T$ is said to be:
\begin{itemize}
		\item[\em (i)] {the Grothendieck operator ({\bf GO}, for short) if every  bounded $w^{*}$-null sequence $({x}'_n)\subseteq F'$, $T'({x}'_n)\xrightarrow{w}0$.}
\item[\em (ii)]
	{ The disjoint Grothendieck operator ({\bf DGO}, in notation) if every  bounded disjoint sequence $({x}'_n)\subseteq F'$, $T'({x}'_n)\xrightarrow{w}0$.}
\item[\em (iii)]
	{The unbounded Grothendieck operator ({\bf UGO}, in notation) if every unbounded $weak^*$-null sequence $({x}'_n)\subseteq F'$,  $T'({x}'_n)\xrightarrow{w}0$.}

			\end{itemize}
\end{definition}
The first part of Definition is initially considered in \cite{DLS}. Furthermore, some different aspects of the Grothendieck property with emphasis on order structure has been studied recently in \cite{Z2}. Now, we consider some relations between these classes of operators.
\begin{proposition}\label{13}
Suppose $E$ and $F$ are Banach lattices. Then we have the following statements.
\begin{itemize}
		\item[\em (i)] {Suppose $F$ has the {\bf GP}. Then every continuous operator $T:E\to F$ is a {\bf GO}.}
\item[\em (ii)]
	{ Suppose $E$ has the {\bf GP} and a continuous operator $T:E\to F$ is a {\bf GO}. Then $F$ has the {\bf GP}, as well.}

				\end{itemize}
\end{proposition}
\begin{proof}
$(i)$. Suppose $({y}'_n)\subseteq F'$ is a $w^*$-null sequence in $F'$. By the assumption, ${y}'_n\xrightarrow{w}0$. This implies that $T'({y}'_n)\xrightarrow{w}0$ by \cite[Theorem 5.22]{AB}.

$(ii)$. Suppose $({y}'_n)\subseteq F'$ is a $w^*$-null sequence in $F'$. By the assumption, $T'({y}'_n)\xrightarrow{w^*}0$ so that $T'({y}'_n)\xrightarrow{w}0$ by the Grothendieck property of $E$.
\end{proof}

By using \cite[Lemma 2]{Z}, we conclude that every {\bf UGO} is {\bf DGO}. For a sort of the converse, we have the following result. Just note that {\bf UGP} is equivalent to the reflexivity; before that, we need the following fact.
\begin{proposition}\label{21}
Suppose $E$ is a Banach lattice and $F$ is an order continuous Banach lattice. Then, every {\bf GO} $T:E\to F$ is {\bf UGO}.
\end{proposition}
\begin{proof}
$(i)$. Suppose $({x'}_n)$ is a norm bounded $uaw^{*}$-null sequence in $F'$. By \cite[Proposition 8]{Z}, ${x'}_n\xrightarrow{w^{*}}0$ so that $T'({x'}_n)\xrightarrow{w}0$ in $E'$ by the assumption.
\end{proof}
\begin{theorem}
A Banach lattice $F$ has the {\bf UGP} if and only if every continuous operator $T:\ell_1\to F$ is an {\bf UGO}.
\end{theorem}
\begin{proof}
Suppose $F$ has the {\bf UGP} and $T:\ell_1\to F$ is a continuous operator. By \cite[Theorem 4]{Z2}, $F$ is reflexive so that it has the {\bf GP}. By Proposition \ref{21}, $T$ is an {\bf UGO}. For the converse, assume that not. Therefore, by \cite[Theorem 4.71]{AB}, $F$ contains a lattice copy of either $\ell_1$ or $c_0$. For the former case, consider the lattice embedding from $\ell_1$ into $F$. We conclude that $\ell_{\infty}$ is also linearly homeomorphic to $F'$. Consider the sequence $(u_n)\subseteq \ell_{\infty}$ defined via $u_n=(0,\ldots,0,1,\ldots)$ in which zero is appeared $n$-times. $u_n\xrightarrow{uaw^{*}}0$ but it is not weakly null by the Dini's theorem, \cite[Theorem 3.52]{AB}. For the latter case, since $c_0$ is linearly homeomorphic to $F$, we see that $\ell_1$ is linearly homeomorphic to $F'$. Note that the standard basic sequence $(e_n)$ is $uaw^{*}$-null in $\ell_1$; nevertheless, it can not be weakly null, certainly.

\end{proof}
Now, we show some ideal properties for the unbounded Grothendieck operators.
\begin{proposition}
Assume that $E$, $F$ and $G$ are Banach lattices. Then, we have the following observations.
\begin{itemize}
		\item[\em (i)]
{If $T:F\to G$ is an {\bf UGO} and $S:E\to F$ is a continuous operator. Then, $TS:E\to G$ is also {\bf UGO}.}
		\item[\em (ii)]
{Suppose $T:E\to F$ is an {\bf UGO} and $S:F\to G$ is interval-preserving. Then, $ST$ is also {\bf UGO} }

\end{itemize}
\end{proposition}
\begin{proof}
$(i)$. Suppose $({x}'_n)$ is a bounded $uaw^*$-null sequence in $G'$. Since $(TS)'=S' T'$, we conclude that $T'({x}'_n)\xrightarrow{w}0$ so that $S' T'({x}'_n)\xrightarrow{w}0$.

$(ii)$. Suppose $({x}'_n)$ is a bounded $uaw^*$-null sequence in $G'$. Since $S$ is interval-preserving, by \cite[Theorem 2.19]{AB}, $S'$ is lattice homomorphism. So, $S'({x}'_n)\xrightarrow{uaw^*}0$. Therefore, $(ST)'({x}'_n)=T'(S'({x}'_n))\xrightarrow{w}0$.
\end{proof}

Now, we have the following fact.
\begin{corollary}\label{41}
A $\sigma$-order complete Banach lattice $F$ is order continuous if and only if every {\bf GO} $T:\ell_{\infty}\to F$ is {\bf UGO}.
\end{corollary}
\begin{proof}
The direct implication follows from \cite[Proposition 8]{Z} since ${\ell_{\infty}}'$ is order continuous. For the other side, suppose not. By \cite[Theorem 4.51]{AB}, there exists a lattice embedding from $\ell_{\infty}$ into $F$. it is {\bf GO}; For, assume that $({x'}_n)$ is a $w^*$-null sequence in $F'$ so that in ${\ell_{\infty}}'$. Since $\ell_{\infty}$ possesses the {\bf GP}, we conclude that ${x'}_n\xrightarrow{w}0$, as claimed. Nevertheless, it can not be an {\bf UGO}. For, suppose $(e_n)$ is the standard basis of $\ell_{\infty}$ which is also disjoint in $F$. By \cite[Lemma 3.4]{AEH}, there exists a disjoint bounded sequence $(g_n)\subseteq F'$ such that $g_n(e_n)=1$. By \cite[Lemma 2]{Z}, $g_n\xrightarrow{uaw}0$ so that $g_n\xrightarrow{uaw^{*}}0$. But, it can not be weakly null in ${\ell_{\infty}}'$ since $\ell_1\subseteq {\ell_{\infty}}'$.
\end{proof}
Note that for Banach lattices $E$ and $F$, every {\bf UGO} $T:E\to F$ is {\bf DGO} by \cite[Lemma 2]{Z}. Apply this with Proposition \ref{21} and Corollary \ref{41}, we conclude the following result.
\begin{corollary}
Suppose $E$ is a Banach lattice and $F$ is a $\sigma$-order complete Banach lattice. Then, we have the following facts.
\begin{itemize}
		\item[\em (i)]
{If $F$ is order continuous, then every {\bf GO} $T:E\to F$ is an {\bf DGO}.}
		\item[\em (ii)]
{If every {\bf GO} $T:\ell_{\infty}\to F$ is an {\bf DGO}, then $F$ is order continuous.}

\end{itemize}
\end{corollary}


\begin{thebibliography}{1}
\bibitem{AB} C.D. Aliprantis and O. Burkinshaw, Positive operators,
Springer, 2006.
\bibitem{AEH} B. Aqzzouz, A. Elbour and J. Hmichane, {\em The duality problem for the Class
of b-weakly compact operators,} Positivity, {\bf 13} (2009), pp. 683--692.
\bibitem{DOT} Y. Deng, M O'Brien, and V. G. Troitsky, {\em Unbounded norm convergence in Banach lattices,} Positivity, {\bf 21(3)} (2017), pp. 963--974.
\bibitem{DLS} P. Domanski, M. Lindstr\"{¨o}m, AND G. Schl\"{u}chtermann, {\em Grothendieck operators on tensor products,} Proceedings of the American mathematical society, {\bf 125(4)} (1997), pp.  2285--2291.

\bibitem{EGZ} N. Erkursun-Ozcan, N. Anıl Gezer, and O. Zabeti, {\em Unbounded absolutely weak Dunford--Pettis operators}, Turkish journal of mathematics, {\bf 43} (2019), pp. 2731--2740.
 \bibitem{GTX} N. Gao, V. G. Troitsky, and F. Xanthos, {\em Uo-convergence and its applications to Ces\`{a}ro means in
Banach lattices}, Israel J. Math, {\bf 220} (2017), pp. 649-–689.
       \bibitem{MFMA} N. Machrafi, K. E. Fahri, M. Moussa and B. Altın{\em A note on weak almost limited operators}, Hacet. J. Math. Stat.,  {\bf 48(3)} (2019), pp. 759--770.
 \bibitem{Ni} P. Nieberg, Banach lattices, Springer-Verlag, Berlin, 1991.
 \bibitem{Z2} O. Zabeti, {\em The Grothendieck property from an ordered point of view}, Positivity,  {\bf 26}(17), (2022), doi: 10.1007/s11117-022-00893-2.
\bibitem{Z} O. Zabeti, {\em Unbounded absolute weak convergence in Banach lattices}, Positivity,  {\bf 22(1)} (2018), pp. 501--505.
\bibitem{Z1} O. Zabeti, {\em Unbounded continuous operators and unbounded Banach-Saks property in Banach lattices}, Positivity,  {\bf 25} (2021), pp. 1989--2001.

\end{thebibliography}
\end{document}